\documentclass[11pt, a4paper]{article}
\usepackage{amsfonts,amsmath,amssymb,latexsym,color,mathrsfs,tikz,breqn,extarrows}
\usepackage{amsthm}
\newtheorem{thm}{Theorem}[section]
\newtheorem{lem}[thm]{Lemma}
\newtheorem{cor}[thm]{Corollary}
\newtheorem{prop}[thm]{Proposition}
\newtheorem{rem}[thm]{Remark}

\usetikzlibrary{calc}

\usepackage{authblk}

\long\def\delete#1{}

\usepackage{cleveref}
\crefformat{section}{\S#2#1#3}
\crefformat{subsection}{\S#2#1#3}
\crefformat{subsubsection}{\S#2#1#3}
\crefrangeformat{section}{\S\S#3#1#4 to~#5#2#6}
\crefmultiformat{section}{\S\S#2#1#3}{ and~#2#1#3}{, #2#1#3}{ and~#2#1#3}

\def\QEDopen{{\hfill\setlength{\fboxsep}{0pt}\setlength{\fboxrule}{0.2pt}\fbox{\rule[0pt]{0pt}{1.3ex}\rule[0pt]{1.3ex}{0pt}}}}

 \newcommand{\BZ}{{\mathbb {Z}}}

\newcommand{\sA}{\mathscr{A}}\newcommand{\sB}{\mathscr{B}}
\newcommand{\sD}{\mathscr{D}}
\newcommand{\sF}{\mathscr{F}}

\newcommand{\sN}{\mathscr{N}}
\newcommand{\sP}{\mathscr{P}}

\newcommand{\sU}{\mathscr{U}}

\newcommand{\floor}[1]{\left\lfloor #1 \right\rfloor}
\newcommand{\ceil}[1]{\left\lceil #1 \right\rceil}

\newcommand{\abs}[1]{\left | #1 \right |}

\newcommand{\ls}{\leqslant}
\newcommand{\gs}{\geqslant}
\newcommand{\cb}[1]{\left\{ #1 \right\}}
\newcommand{\pt}[1]{\left(#1\right)}

\begin{document}
\setcounter{page}{1}

\title{\bf Erd\H{o}s-Ko-Rado Theorem for Bounded Multisets}

\author[1]{
Jiaqi Liao\thanks{E-mail:\texttt{liaojq19@mails.tsinghua.edu.cn}}}
\author[1]{
Zequn Lv\thanks{E-mail:\texttt{lvzq19@mails.tsinghua.edu.cn}}}
\author[2]{
Mengyu Cao\thanks{E-mail:\texttt{myucao@ruc.edu.cn}}}
\author[1]{
Mei Lu\thanks{E-mail:\texttt{lumei@tsinghua.edu.cn}}}

\affil[1]{\small Department of Mathematical Sciences, Tsinghua University, Beijing 100084, China}
\affil[2]{\small Institute for Mathematical Sciences, Renmin University of China, Beijing 100086, China}

\date{}
\openup 0.5\jot
\maketitle

\begin{abstract}
Let $ k, m, n $ be positive integers with $ k \gs 2 $. A $ k $-multiset of $ [n]_m $ is a collection of $ k $ integers from the set $ \cb{1, 2, \ldots, n} $ in which the integers can appear more than once but at most $ m $ times. A family of such $ k $-multisets is called an intersecting family if every pair of $ k $-multisets from the family have non-empty intersection. A finite sequence of real numbers $\{a_1,a_2,\ldots,a_n\}$ is said to be unimodal if there is some $k\in \{1,2,\ldots,n\}$,
 such that $a_1\ls a_2\ls\ldots\ls a_{k-1}\ls a_k\gs a_{k+1}\gs \ldots\gs a_n$. Given $m,n,k$, denote $C_{k,l}$ as the coefficient of $x^k$ in the generating function $(\sum_{i=1}^mx^i)^l$, where $1\ls l\ls n$. In this paper, we first show that the sequence of $\{C_{k,1},C_{k,2},\ldots,C_{k,n}\}$ is unimodal. Then we use this as a tool to prove
  that the intersecting family in which every $ k $-multiset contains a fixed element attains the maximum cardinality for $ n \gs k + \ceil{k/m} $. In the special case when $m = 1$ and $m=\infty$, our result gives rise to the famous Erd\H{o}s-Ko-Rado Theorem and an unbounded multiset version for this problem given by Meagher and Purdy \cite{MP11}, respectively. The main result in this paper can be viewed as a bounded multiset version of the Erd\H{o}s-Ko-Rado Theorem.
	
\vspace{2mm}
	
	
\
	
\end{abstract}

\section{Introduction}

\subsection{Notations}

Let $ m, n $ be positive integers ($ m = \infty $ is allowed). Denote  $ [n]_m := \cb{m \cdot 1, m \cdot 2, \ldots, m \cdot n} $, that is, $ [n]_m $ contains exactly $ m $ symbols $ i $ for each $ i = 1, 2, \ldots, n $. Then $[n]_1=[n]:= \cb{1, 2, \ldots, n} $. When $ m < \infty $, $ [n]_m $ is called \emph{\textbf{bounded}}, and  $ [n]_\infty $ is called \emph{\textbf{unbounded}}.
Let $ m_i $ be  integer, where $1\ls i\ls n$. We call $ \{m_1 \cdot 1,$ $ m_2 \cdot 2, \ldots, m_n \cdot n\} $ a \emph{\textbf{multiset}} of $ [n]_m $ if $0\ls m_i\ls m$ for all $1\ls i\ls n$. Note that  the cardinality of a multiset
is the total number of elements including repetitions.

 Let $ k $ be a positive integer. Denote $ \binom{[n]_m}{k} := \cb{A \subseteq [n]_m: \abs{A} = k} $. Let $ t $ be a positive integer and $ s $ a non-negative integer with $ s \ls k - t $, denote

\[\sF_{s, t}^{(m)} := \cb{A \in \binom{[n]_m}{k}: \abs{A \cap [2s + t]} \gs s + t}.\]

\noindent A family $ \sA \subseteq \binom{[n]_m}{k} $ is called  \textbf{\emph{$ t $-intersecting}} if for any $ A_1, A_2 \in \sA $, we have $ \abs{A_1 \cap A_2} \gs t $. For simplicity, ``intersecting" means ``$ 1 $-intersecting". Two families $ \sA_1, \sA_2 \subseteq \binom{[n]_m}{k} $ are called \textbf{\emph{isomorphic}} if there is a permutation $ \sigma $ on $ [n] $ such that $ \sA_2 = \sigma(\sA_1) := \cb{\sigma(A): A \in \sA_1} $ and denoted by $ \sA_1 \cong \sA_2 $. An intersecting family $ \sA \subseteq \binom{[n]_m}{k} $ is called \emph{\textbf{trivial}} if $ \sA \subseteq \sF_{0, 1}^{(m)} $ up to isomorphism, and \emph{\textbf{non-trivial}} otherwise.

\subsection{Background}

One of the basic problems in extremal set theory is to decide how large these intersecting families can be, and to describe the structure of the intersecting families that meet whatever bound we can derive.

\subsubsection{Sets}

The famous Erd\H{o}s-Ko-Rado Theorem \cite{EKR61} is the first result in extremal set theory that gives the size and structure of the largest intersecting family in $ \binom{[n]}{k} $.

\begin{thm}[Erd\H{o}s, Ko and Rado \cite{EKR61}]\label{EKR_Theorem}

Let $ k, n $ be positive integers with $ k \gs 2 $ and $ n \gs 2k $. If $ \sA \subseteq \binom{[n]}{k} $ is an intersecting family, then $ \abs{\sA} \ls \abs{\sF_{0, 1}^{(1)}} $. Moreover, when $ n > 2k $, equality holds if and only if $ \sA \cong \sF_{0, 1}^{(1)} $.

\end{thm}

A second theorem  in \cite{EKR61}  showed that each extremal $t$-intersecting family of ${[n]\choose k}$ consists of all $k$-subsets that contain a fixed $t$-subset of $[n]$ for $n>n_0(k,t)$. It is known that the smallest possible such function $n_0(k, t)$ is $(t+1)(k-t+1).$ This was proved by Frankl \cite{Frankl-1978} for $t\geq 15$ and subsequently determined by Wilson \cite{Wilson-1984} for all $t$. In \cite{Frankl-1978}, Frankl also made a conjecture on the maximum size of a $t$-intersecting family of ${[n]\choose k}$ for all positive integers $t,k$ and $n$. This conjecture was partially proved by Frankl and F\"{u}redi in  \cite{Frankl--Furedi-1991} and completely settled by Ahlswede and Khachatrian in \cite{AK97}. 

\begin{thm}[Ahlswede and Khachatrian \cite{AK97}]\label{AK_Theorem}

Let $ k, n, t $ be positive integers with $ n \gs k \gs \max\{2, t\} $ and $ s $ a non-negative integer with $ s \ls k - t $.

\begin{itemize}
	\item [\rm(a)] If $ (k - t + 1)\pt{2 + \frac{t - 1}{s + 1}} < n < (k - t + 1)\pt{2 + \frac{t - 1}{s}} $, then $ \sF_{s, t}^{(1)} $ is the unique (up to isomorphism) $ t $-intersecting family in $ \binom{[n]}{k} $ with maximum cardinality. $($By convention, $ \frac{t - 1}{s} = \infty $ for $ s = 0 $.$)$
	
	\item [\rm(b)] If $ n = (k - t + 1)\pt{2 + \frac{t - 1}{s + 1}} $, then $ \sF_{s, t}^{(1)}, \sF_{s + 1, t}^{(1)} $ are the unique $($up to isomorphism$)$ $ t $-intersecting families in $ \binom{[n]}{k} $ with maximum cardinality.
\end{itemize}

\end{thm}

\subsubsection{Unbounded multisets}

In this paper, we focus on intersection problems for multisets. Meagher and Purdy were the first authors to give an unbounded multiset version of the Erd\H{o}s-Ko-Rado Theorem by using the graph homomorphism method \cite{MP11}.

\begin{thm}[Meagher and Purdy \cite{MP11}]\label{MP_Theorem}
	
Let $ k, n $ be positive integers with $ k \gs 2 $ and $ n \gs k + 1 $. If $ \sA \subseteq \binom{[n]_\infty}{k} $ is an intersecting family, then $ \abs{\sA} \ls \abs{\sF_{0, 1}^{(\infty)}} $. Moreover, when $ n > k + 1 $, equality holds if and only if $ \sA \cong \sF_{0, 1}^{(\infty)} $.

\end{thm}

The following theorem is an unbounded multiset version of the complete Erd\H{o}s-Ko-Rado Theorem. In \cite{FGV15}, F\"{u}redi, Gerbner and Vizer gave the size of the largest $t$-intersecting families of $k$-multisets from $[n]_{\infty}$ using an operation which they call down-compression. In \cite{MP16}, Meagher and Purdy used the down-compression operation from \cite{FGV15} and a graph homomorphism to prove the further results and gave the structure of the families that attain the maximum size.

\begin{thm}[F\"{u}redi, Gerbner and Vizer \cite{FGV15}; Meagher and Purdy \cite{MP16}]\label{FGVMP_Theorem}
	
Let $ k, n, t $ be positive integers with $ k \gs \max\{2, t\} $, $ n \gs 2k - t $ and $ s $ a non-negative integer with $ s \ls k - t $. Let $ \sA \subseteq \binom{[n]_\infty}{k} $ be an $t$-intersecting family.

\begin{itemize}
\item [\rm(a)] If $ (k - t + 1)\pt{2 + \frac{t - 1}{s + 1}} < n + k - 1 < (k - t + 1)\pt{2 + \frac{t - 1}{s}} $, then $ \abs{\sA} \ls \abs{\sF_{s, t}^{(\infty)}} $. (By convention, $ \frac{t - 1}{s} = \infty $ for $ s = 0 $.) If $ s > 0 $, equality holds if and only if $ \sA \cong \sF_{s, t}^{(\infty)} $. If $ s = 0 $, equality holds if and only if $ \sA $ consists of all $ k $-multisets containing a fixed $ t $-multiset.
	
\item [\rm(b)] If $ n + k - 1 = (k - t + 1)\pt{2 + \frac{t - 1}{s + 1}} $, then $ \abs{\sA} \ls \abs{\sF_{s, t}^{(\infty)}} = \abs{\sF_{s + 1, t}^{(\infty)}} $. If $ s > 0 $, equality holds if and only if $ \sA \cong \sF_{s, t}^{(\infty)} $ or $ \sF_{s + 1, t}^{(\infty)} $. If $ s = 0 $, equality holds if and only if $ \sA \cong \sF_{1, t}^{(\infty)} $ or it consists of all $ k $-multisets containing a fixed $ t $-multiset.
\end{itemize}	

\end{thm}

For more on intersection theorems of finite sets, see \cite{CLLW22, F78}. For more on extremal non-trivial intersecting families of finite sets, see \cite{AK96, HM67}. For more on intersection theorems in other categories, see \cite{GM16}.

\subsection{Main result}

In this paper, we present a bounded multiset version of the Erd\H{o}s-Ko-Rado Theorem.

\begin{thm}\label{main_thm}
Let $ k, m, n $ be positive integers with $ k \gs 2 $ and $ n \gs k + \ceil{k/m} $. If $ \sA \subseteq \binom{[n]_m}{k} $ is an intersecting family, then $ \abs{\sA} \ls \abs{\sF_{0, 1}^{(m)}} $. Moreover, when {\rm(a)} $ n > k + \ceil{k/m} $, or {\rm(b)} $ n = k + \ceil{k/m} $, $ k > m $ and $ m \nmid k $, equality holds if and only if $ \sA \cong \sF_{0, 1}^{(m)} $.
\end{thm}

In the special case when $m = 1$ and $m=\infty$ in Theorem \ref{main_thm}, our result gives rise to Theorem 1.1 (the famous Erd\H{o}s-Ko-Rado Theorem \cite{EKR61}) and Theorem 1.3 (the main result in \cite{MP11} by Meagher and Purdy), respectively.

\begin{rem}
\label{rem1}
{\em
There is a vast, excellent literature on determining the extremal non-trivial $t$-intersecting families for finite sets. The first result was the Hilton-Milner Theorem \cite{HM67}. In \cite{MP16}, Meagher and Purdy gave  an unbounded multiset version of the Hilton-Milner Theorem. To limit the scope of this paper, we will not introduce them here. In another paper, we will study the bounded multiset version of Hilton-Milner Theorem.}
\end{rem}

A finite sequence of real numbers $\{a_1,a_2,\ldots,a_n\}$ is said to be \textbf{\emph{unimodal}} if there is some $k\in \{1,2,\ldots,n\}$, such that $a_1\le a_2\le\ldots\le a_{k-1}\le a_k\ge a_{k+1}\ge \ldots\ge a_n$. Unimodality problems often arise in many branches of mathematics and have been extensively investigated. Given $m,n,k$, denote $C_{k,l}$ as the coefficient of $x^k$ in the generating function $(\sum_{i=1}^mx^i)^l$, where $1\ls l\ls n$. In this paper, we first show that the sequence of $\{C_{k,1},C_{k,2},\ldots,C_{k,n}\}$ is unimodal in Section \ref{pre}. Then we use this as a tool to prove our main result.

This paper is organized as follows. In Section \ref{pre}, we prove two properties of $C_{k,l}$, which are the key in the proof of Theorem \ref{main_thm}. In Section \ref{proof}, we prove Theorem \ref{main_thm}.

\section{Preliminaries}\label{pre}
Let $ k, m, \ell $ be positive integers with $ m \gs 2 $. Recall that  $ C_{k, \ell} $ is the coefficient of $ x^k $ in the generating function $ \pt{\sum_{i = 1}^{m} x^i}^\ell $. Then $ \pt{\sum_{i = 1}^{m} x^i}^\ell = \sum_{k = 1}^{\infty} C_{k, \ell} \cdot x^k $. We omit $ ``m" $ in the notation since once $ m $ is given, it is fixed throughout the context. For later use, we always denote $ q := \ceil{k/m} $ throughout this paper. We collect some basic facts on $ C_{k, \ell} $ in the following lemma. For convenience, we set $C_{k,\ell}=0$ if $\ell\ls 0$.

\begin{lem}\label{easy_lemma}
$ C_{k, \ell} $ satisfies the following properties.
	
\begin{enumerate}		
\item\label{easy_property1} $ C_{k, \ell} \ne 0 $ if and only if $ q \ls \ell \ls k $.
		
\item\label{easy_property2} $ C_{k, q} = 1 $ if and only if {\rm(a)} $ q = 1 $, or {\rm(b)} $ q > 1 $ and $ m \mid k $.
		
\item\label{easy_property3} $ C_{k, k} \equiv 1 $.
		
\item\label{easy_property4} When $ q = 1 $, we have $ C_{k, \ell} = \binom{k - 1}{\ell - 1} $.

\item\label{easy_property5} $ C_{k, \ell} = C_{k - 1, \ell - 1} + C_{k - 2, \ell - 1} + \cdots + C_{k - m, \ell - 1} $ if $\ell\gs 2$.
\end{enumerate}	
\end{lem}

\begin{proof}

Note that $ C_{k, \ell} $ has the following combinatorical interpretation.

\[C_{k, \ell} = \abs{\cb{(m_1, \ldots, m_\ell) \in [m]^\ell: m_1 + \cdots + m_\ell = k}}.\]

{\rm(\romannumeral1)} Note that the monomial of minimum degree in $ \pt{\sum_{i = 1}^{m} x^i}^\ell $ is $ x^\ell $, and the monomial of maximum degree in $ \pt{\sum_{i = 1}^{m} x^i}^\ell $ is $ x^{m\ell} $. According to the definition of $ C_{k, \ell} $, we have $ C_{k, \ell} \ne 0 $ if and only if $ \ell \ls k \ls m\ell $, which is equivalent to $ q \ls \ell \ls k $.

{\rm(\romannumeral2)} and {\rm(\romannumeral3)} follow from the combinatorical interpretation of $ C_{k, \ell} $.

{\rm(\romannumeral4)} Note that $ q = 1 $ means $ k \ls m $. By the combinatorical interpretation of $ C_{k, \ell} $, this case is equivalent to $ m = \infty $. Thus $ C_{k, \ell} = \binom{\ell + k - \ell - 1}{k - \ell} = \binom{k - 1}{\ell - 1} $.

{\rm(\romannumeral5)} On one hand, we have

\[\pt{\sum_{i = 1}^{m} x^i}^\ell = \sum_{k = \ell}^{m\ell} C_{k, \ell} \cdot x^k.\]

\noindent On the other hand, we have

\begin{align*}
\pt{\sum_{i = 1}^{m}x^i}^\ell
{}&{}={}\pt{\sum_{i = 1}^{m} x^i}^{\ell - 1} \times \pt{\sum_{i = 1}^{m} x^i} \\
{}&{}={}\pt{\sum_{k = \ell - 1}^{m(\ell - 1)}C_{k, \ell - 1} \cdot x^k} \times \pt{\sum_{i = 1}^{m} x^i} \\
{}&{}={}\sum_{k = \ell}^{m\ell} \pt{C_{k - 1, \ell - 1} + C_{k - 2, \ell - 1} + \cdots + C_{k - m, \ell - 1}} \cdot x^k. \qedhere
\end{align*}

\end{proof}

Let $S$ be a finite set consisting of  positive integers. We use $\min S$ (resp. $\max S$) the denote the minimum integer (resp. the maximum integer) in $S$. Recall $q=\ceil{k/m}$.
Define

\[\alpha(k) := \min\cb{i: C_{k, i} = \max\cb{C_{k, q}, \ldots, C_{k, k}}}.\]

\noindent Then $\alpha(k)\gs q\gs 1$. By Lemma \ref{easy_lemma} (\ref{easy_property1}), we have $ C_{k, \alpha(k)} > 0 $. 

\begin{prop}\label{unimodality}
Let $ k $ be a positive integer. We have

\begin{itemize}
\item[\rm(a)] $ 0 \ls \alpha(k) - \alpha(k - 1) \ls 1 $.
	
\item[\rm(b)] $C_{k,q}\ls \cdots \ls C_{k, \alpha(k) - 1} \ls C_{k, \alpha(k)} \gs C_{k, \alpha(k) + 1} \gs \cdots \gs C_{k,k}$.
\end{itemize}
\end{prop}

\begin{proof}
We prove it by induction on $ k $. If $1\ls k\ls 2$, then $q=\ceil{k/m}=1$ by  $m\gs 2$. By Lemma \ref{easy_lemma} (ii) and (iii), $ C_{1, 1} = C_{2, 1} = C_{2, 2} = 1 $. So $ \alpha(1) = \alpha(2) = 1 $, which implies the proposition holds for $ k \ls 2 $. When $ q = 1 $, by Lemma \ref{easy_lemma} (\ref{easy_property4}), this proposition holds according to the basic properties of binomial coefficients. Thus we may assume that $ k > m $ which implies $q\ge 2$. To prove that the proposition holds for $ k $, it suffices to show the following two statements holds.
	
\begin{enumerate}
\item\label{first_ineq} $ C_{k, \alpha(k - 1) + 1} \gs C_{k, \alpha(k - 1) + 2} \gs \cdots \gs C_{k,k}$.
		
\item\label{second_ineq} $ C_{k, \alpha(k - 1)} \gs C_{k, \alpha(k - 1) - 1} \gs \cdots \gs C_{k,q}$.
\end{enumerate}

\noindent In fact, if (\ref{first_ineq}) and (\ref{second_ineq}) hold simultaneously, then $\max\{C_{k, q}, \ldots, C_{k, k}\}\in \{ C_{k, \alpha(k - 1)} , C_{k, \alpha(k - 1) + 1}\} $. Hence $ \alpha(k) \in \{\alpha(k - 1), \alpha(k - 1) + 1 \}$. Thus {\rm(b)} holds.
	
(\ref{first_ineq}). Let $ s \gs \alpha(k - 1) + 1 $. By inductive hypothesis on {\rm(a)}, we have

\[s > s - 1 \gs \alpha(k - 1) \gs \alpha(k - 2) \gs \cdots \gs \alpha(k - m).\]

\noindent By Lemma \ref{easy_lemma} (\ref{easy_property5}), we have
		
\begin{align*}
C_{k, s} - C_{k, s + 1}{}={}&\sum_{i = k - m}^{k - 1}C_{i, s - 1} - \sum_{i = k - m}^{k - 1}C_{i, s} \\
={}&\sum_{i = k - m}^{k - 1}\pt{C_{i, s - 1} - C_{i, s}} \\
\gs{}&0.
\end{align*}

\noindent The last inequality holds by inductive hypothesis on {\rm(b)}.

(\ref{second_ineq}). Let $ s \ls \alpha(k - 1) $. We will show that $C_{k, s} \gs C_{k, s - 1}$. If $s=1$, the result holds obviously by $C_{k,0}=0$. Note that $C_{k,1}=0$ by Lemma \ref{easy_lemma} (\ref{easy_property1}) and $q\gs 2$. So the result holds if $s=2$. Thus we assume $s\ge 3$ and discuss it in two cases.
		
\begin{itemize}
\item[Case \rm1.] $ C_{k - m - 1, s - 2} > C_{k - m - 1, s - 1} $. By Lemma \ref{easy_lemma} (\ref{easy_property5}), we have
			
\begin{align*}
C_{k, s} - C_{k, s - 1}{}={}&\pt{\sum_{i = k - m}^{k - 1}C_{i, s - 1}} - C_{k, s - 1} \\
={}&\pt{\sum_{i = k - m - 1}^{k - 2} C_{i, s - 1}} - C_{k - m - 1, s - 1} + C_{k - 1, s - 1} - C_{k, s - 1} \\
={}&C_{k - 1, s} - C_{k - m - 1, s - 1} + \pt{C_{k - 1, s - 1} - C_{k, s - 1}} \\
={}&C_{k - 1, s} - C_{k - m - 1, s - 1} + \pt{\sum_{i = k - m - 1}^{k - 2}C_{i, s - 2} - \sum_{i = k - m}^{k - 1}C_{i, s - 2}} \\
={}&C_{k - 1, s} - C_{k - m - 1, s - 1} + C_{k - m - 1, s - 2} - C_{k - 1, s - 2} \\
={}&\pt{C_{k - m - 1, s - 2} - C_{k - m - 1, s - 1}} + \pt{C_{k - 1, s} - C_{k - 1, s - 2}} \\
\gs{}&0.
\end{align*}

\noindent The last inequality holds by inductive hypothesis on {\rm(b)} and the condition $ C_{k - m - 1, s - 2} > C_{k - m - 1, s - 1} $.
			
\item[Case \rm2.] $ C_{k - m - 1, s - 2} \ls C_{k - m - 1, s - 1} $. In this case, $ \alpha(k - m - 1) \gs s - 1 $. By inductive hypothesis on {\rm(a)}, $ \alpha(k - 1) \gs \alpha(k - 2) \gs \ldots \gs \alpha(k - m - 1) \gs s -1 $. By Lemma \ref{easy_lemma} (\ref{easy_property5}) again, we have
			
\begin{align*}
C_{k, s} - C_{k, s - 1}{}={}&\pt{\sum_{i = k - m}^{k - 1}C_{i, s - 1}} - \pt{\sum_{i = k - m}^{k - 1}C_{i, s - 2}} \\
={}&\sum_{i = k - m}^{k - 1} \pt{C_{i, s - 1} - C_{i, s - 2}} \\
\gs{}&0.
\end{align*}

\noindent The last inequality holds by inductive hypothesis on {\rm(b)}. \qedhere

\end{itemize}
\end{proof}
By Proposition \ref{unimodality}, the sequence of $\{C_{k,1},C_{k,2},\ldots,C_{k,n}\}$ is unimodal.
Recall the definition of the binomial coefficient, $ \binom{a}{b} = 0 $ if $ a < \max\{0, b\} $. Let $ j, \ell, r $ be positive integers with $ 1 \ls j < r \ls m $, denote

\[S_{j, \ell, r} := \cb{i: i \in \BZ, 0 \ls i \ls r - 1 \text{ and } \binom{r - 1}{i} - \binom{r - j - 1}{i} \gs \ell}.\]
Then $\min S_{j, \ell, r}\gs 0$ and $\max S_{j, \ell, r}\ls r-1$. We have the follpwing result.
\begin{lem}\label{consecutive}
$ S_{j, \ell, r} $ is a sequence of consecutive integers.
\end{lem}

\begin{proof}
Define

\[f(i) := \binom{r - 1}{i} - \binom{r - j - 1}{i} .\]
Then $f(i)= \sum_{p = r - j - 1}^{r - 2}\binom{p}{i - 1}$ from the recursive formula of the binomial coefficient.
	
\begin{enumerate}
\item When $ 0 \ls i \ls \ceil{\frac{r - j - 1}{2}} $, we have
		
\[f(i) = \sum_{p = r - j - 1}^{r - 2}\binom{p}{i - 1} \ls \sum_{p = r - j - 1}^{r - 2}\binom{p}{i} = f(i + 1).\]
		
\item When $ \ceil{\frac{r - j - 1}{2}} < i < \ceil{\frac{r - 1}{2}} $, we have
		
\[f(i) = \binom{r - 1}{i} - \binom{r - j - 1}{i} \ls \binom{r - 1}{i + 1} - \binom{r - j - 1}{i + 1} = f(i + 1).\]
		
\item When $ \ceil{\frac{r - 1}{2}} \ls i \ls r - 1 $, we have
		
\[f(i) = \sum_{p = r - j - 1}^{r - 2}\binom{p}{i - 1} \gs \sum_{p = r - j - 1}^{r - 2}\binom{p}{i} = f(i + 1).\]
\end{enumerate}
	
\noindent Hence $ \{f(i):~0\ls i\ls r-1\}$ is unimodal, and we may assume there is an integer $ \beta $ with $ 0 \ls \beta \ls r - 1 $ such that

\[f(0) \ls \cdots \ls f(\beta - 1) \ls f(\beta) \gs f(\beta + 1) \gs \cdots \gs f(r - 1).\]

\noindent If $ i_0 \in S_{j, \ell, r} $ with $ i_0 \ls \beta $, then $ f(\beta) \gs f(\beta - 1) \gs \cdots \gs f(i_0) \gs \ell $, which implies $ [i_0, \beta] \subseteq S_{j, \ell, r} $. If $ i_0' \in S_{j, \ell, r} $ with $ i_0' \gs \beta $, then $ f(\beta) \gs f(\beta + 1) \gs \cdots \gs f(i_0') \gs \ell $, which implies $ [\beta, i_0'] \subseteq S_{j, \ell, r} $. It implies that $ S_{j, \ell, r} $ is a sequence of consecutive natural numbers.
\end{proof}

\begin{lem}\label{transform}
Let  $ 1 \ls r \ls m $ and $s\gs 2$. We have

\begin{equation}\label{eqn2}
C_{(q - 1)m + r, s} = \sum_{j = 1}^{m}\sum_{\ell = 1}^{\infty}\sum_{i \in S_{j, \ell, r}}C_{(q - 2)m + j, s - i - 1}.
\end{equation}

\end{lem}

\begin{proof}
 We  first prove that the following equality holds by induction on $ r $ and $s$.
{\small
\begin{equation}\label{eqn1}	
C_{(q - 1)m + r, s} = \sum_{j = 1}^{m} \sum_{i = 0}^{r - 1} \pt{\binom{r - 1}{i} - \binom{r - j - 1}{i}}C_{(q - 2)m + j, s - i - 1}.	
\end{equation}	
}
\noindent If $s=2$, then $C_{(q - 2)m + j, s - i - 1}\not=0$ if and only if $i=0$. Thus (\ref{eqn1}) holds for $ s = 2 $ by Lemma \ref{easy_lemma} (\ref{easy_property5}). By Lemma \ref{easy_lemma} (\ref{easy_property5}), (\ref{eqn1}) holds for $ r = 1 $. Assume (\ref{eqn1}) holds for $ r \ls m - 1 $ and $s\gs 3$. By Lemma \ref{easy_lemma}  (\ref{easy_property5}), we have
{\small
\begin{align}
C_{(q - 1)m + r + 1, s}{}={}&\sum_{i = r - m + 1}^{r}C_{(q - 1)m + i, s - 1}\notag \\
={}&C_{(q - 1)m + r, s - 1} + \pt{\sum_{i = r - m}^{r - 1}C_{(q - 1)m + i, s - 1}} - C_{(q - 2)m + r, s - 1}\notag \\
={}&C_{(q - 1)m + r, s - 1} + C_{(q - 1)m + r, s} - C_{(q - 2)m + r, s - 1}.\label{subs1}
\end{align}
}
\noindent By inductive hypothesis, we have
{\small
\begin{align*}
C_{(q - 1)m + r, s - 1}{}={}&\sum_{j = 1}^{m}\sum_{i = 0}^{r - 1}\pt{\binom{r - 1}{i} - \binom{r - j - 1}{i}}C_{(q - 2)m + j, s - i - 2} \\
={}&\sum_{j = 1}^{m}\sum_{i = 1}^{r}\pt{\binom{r - 1}{i - 1} - \binom{r - j - 1}{i - 1}}C_{(q - 2)m + j, s - i - 1}
\end{align*}}

\noindent and
{\small
\[C_{(q - 1)m + r, s} = \sum_{j = 1}^{m}\sum_{i = 0}^{r - 1}\pt{\binom{r - 1}{i} - \binom{r - j - 1}{i}}C_{(q - 2)m + j, s - i - 1}.\]
}
\noindent Hence we have
{\small
\begin{align*}
C_{(q - 1)m + r + 1, s}{}=&C_{(q - 1)m + r, s - 1} + C_{(q - 1)m + r, s} - C_{(q - 2)m + r, s - 1} \\
=&{}\sum_{j = 1}^{m}\sum_{i = 1}^{r}\pt{\binom{r - 1}{i - 1} - \binom{r - j - 1}{i - 1}}C_{(q - 2)m + j, s - i - 1} \\
&{}+ \sum_{j = 1}^{m}\sum_{i = 0}^{r - 1}\pt{\binom{r - 1}{i} - \binom{r - j - 1}{i}}C_{(q - 2)m + j, s - i - 1} - C_{(q - 2)m + r, s - 1} \\
=&{}\sum_{j = 1}^{m}\sum_{i = 1}^{r - 1}\pt{\binom{r - 1}{i - 1} - \binom{r - j - 1}{i - 1} + \binom{r - 1}{i} - \binom{r - j - 1}{i}}C_{(q - 2)m + j, s - i - 1} \\
&{}+ \sum_{j = 1}^{m}C_{(q - 2)m + j, s - r - 1} + \sum_{j = 1}^{m}\pt{1 - \binom{r - j - 1}{0}}C_{(q - 2)m + j, s - 1} - C_{(q - 2)m + r, s - 1}\\
=&{}\sum_{j = 1}^{m}\sum_{i = 1}^{r - 1}\pt{\binom{r}{i} - \binom{r - j}{i}}C_{(q - 2)m + j, s - i - 1} + \sum_{j = 1}^{m}C_{(q - 2)m + j, s - r - 1} \\
&{}+ \sum_{j = 1}^{m}\pt{1 - \binom{r - j - 1}{0}}C_{(q - 2)m + j, s - 1} - C_{(q - 2)m + r, s - 1}
\end{align*}
\begin{align*}
=&{}\sum_{j = 1}^{m}\sum_{i = 1}^{r - 1}\pt{\binom{r}{i} - \binom{r - j}{i}}C_{(q - 2)m + j, s - i - 1} + \sum_{j = 1}^{m}C_{(q - 2)m + j, s - r - 1}\\
&{}+ \pt{\sum_{j = r}^{m}C_{(q - 2)m + j, s - 1} - C_{(q - 2)m + r, s - 1}}\\
=&{}\sum_{j = 1}^{m}\sum_{i = 1}^{r - 1}\pt{\binom{r}{i} - \binom{r - j}{i}}C_{(q - 2)m + j, s - i - 1} + \sum_{j = 1}^{m}C_{(q - 2)m + j, s - r - 1} + \sum_{j = r + 1}^{m}C_{(q - 2)m + j, s - 1} \\
=&{}\sum_{j = 1}^{m}\sum_{i = 0}^{r}\pt{\binom{r}{i} - \binom{r - j}{i}}C_{(q - 2)m + j, s - i - 1}.
\end{align*}}

Thus (\ref{eqn1}) holds. Now we prove (\ref{eqn2}) by doing summation by parts.

{\small
\begin{align*}
&\sum_{i = 0}^{r - 1}\pt{\binom{r - 1}{i} - \binom{r - j - 1}{i}}C_{(q - 2)m + j, s - i - 1} \\
={}&\sum_{\ell = 1}^{\infty} \abs{\cb{i \in \BZ: 0 \ls i \ls r - 1, \binom{r - 1}{i} - \binom{r - j - 1}{i} = \ell}} \cdot \ell \cdot C_{(q - 2)m + j, s - i - 1} \\
={}&\sum_{\ell = 1}^{\infty}\pt{\sum_{i \in S_{j, \ell, r}} \ell \cdot C_{(q - 2)m + j, s - i - 1} - \sum_{i \in S_{j, \ell + 1, r}} \ell \cdot C_{(q - 2)m + j, s - i - 1}} \\
={}&\sum_{\ell = 1}^{\infty}\pt{\sum_{i \in S_{j, \ell, r}} \ell \cdot C_{(q - 2)m + j, s - i - 1} - \sum_{i \in S_{j, \ell + 1, r}} (\ell + 1) \cdot C_{(q - 2)m + j, s - i - 1} + \sum_{i \in S_{j, \ell + 1, r}} C_{(q - 2)m + j, s - i - 1}} \\
={}&\sum_{\ell = 1}^{\infty}\sum_{i \in S_{j, \ell, r}}C_{(q - 2)m + j, s - i - 1}.
\end{align*}}

\noindent Thus the result follows from (\ref{eqn1}).
\end{proof}


\begin{lem}\label{equiv}
Let $ k, m $ be positive integers with $ k, m \gs 2 $, denote $ q := \ceil{k/m} $. The following two statements are equivalent.

\begin{itemize}
\item[\rm(a)]\label{equiv1} For any  integer $ d $ with $ 0\ls d \ls \frac{1}{2}\pt{k - q} $, we have

\[C_{k, q + d} \gs C_{k, k - d}.\]

\item[\rm(b)]\label{equiv2} Let $ S_1 $ and $ S_2 $ be two finite sets consisting of consecutive positive integers. If

\[\abs{S_1} = \abs{S_2}, \min S_1 + \max S_2 \gs k + q \text{ and } \min S_1 \ls \min S_2,\]

\noindent then we have

\[\sum_{i \in S_1} C_{k, i} \gs \sum_{i \in S_2} C_{k, i}.\]	
\end{itemize}
\end{lem}

\begin{proof}
If we take $ S_1 = \{q + d\} $ and $ S_2 = \cb{k - d} $ for any  integer $ d $ with $ 0\ls d \ls \frac{1}{2}\pt{k + q} $, then we have {\rm(b)} implies {\rm(a)}.

Now we prove that {\rm(a)} implies {\rm(b)}. Denote $ m_i := \min S_i $ and $ M_i := \max S_i $. Then $m_1+M_2\gs k+q$. We may assume $ M_1 < m_2 $; otherwise we just need to delete the same terms of both sets. Suppose {\rm(a)} holds.   Then
\[\max\cb{C_{k, q}, \ldots, C_{k, k}} = \max\cb{C_{k, q}, \ldots, C_{k, q + d_0}}\]

\noindent where $ d_0 = \floor{\frac{1}{2}\pt{k - q}}$. Thus $ \alpha(k) \ls q + d_0 \ls \frac{1}{2}\pt{k + q} $.	Note that $ \abs{S_1} = \abs{S_2} $ implies $ M_1 - m_1 = M_2 - m_2 =: \mu $. Thus we have

\[\sum_{i \in S_1}C_{k, i} - \sum_{i \in S_2}C_{k, i} = \sum_{i = 0}^{\mu}\pt{C_{k, m_1 + i} - C_{k, M_2 - i}}.\]

\begin{enumerate}
\item If $ m_1 + i < \frac{1}{2}\pt{k + q} $ for some $0\ls i\ls \mu$, then $ 2(m_1 + i - q) < k - q $. By (a), we have $C_{k, m_1 + i} = C_{k, q + (m_1 + i - q)} \gs C_{k, k - (m_1 + i - q)} = C_{k, k + q - m_1 - i}$. Since $m_1+M_2\gs k+q$, we also have $ M_2 - i \gs k + q - m_1 - i > \frac{1}{2}\pt{k + q} \gs \alpha(k) $. By Proposition \ref{unimodality} {\rm(b)}, we have $C_{k, k + q - m_1 - i} \gs C_{k, M_2 - i}$. Thus we have $C_{k, m_1 + i}\gs C_{k, M_2 - i}$.




\item If $ m_1 + i \gs \frac{1}{2}\pt{k + q} $  for some $0\ls i\ls \mu$, then $ M_2 - i \gs m_2 > M_1 \gs m_1 + i \gs \alpha(k) $. By Proposition \ref{unimodality} {\rm(b)}, we have $C_{k, m_1 + i} \gs C_{k, M_2 - i}.$

\end{enumerate} By (i) and (ii), we have $\sum_{i \in S_1}C_{k, i} \gs \sum_{i \in S_2}C_{k, i}$.
\end{proof}

\begin{prop}\label{weak_spirality}
	
Let $ k, m $ be positive integers with $ k, m \gs 2 $. For any positive integer $ d $ with $ 2d \ls k - \ceil{k/m} $, we have
\[C_{k, \ceil{k/m} + d} \gs C_{k, k - d}.\]
\end{prop}

\begin{proof} We prove the result  by induction on $ k $. If $ k\ls m $,  then $\ceil{k/m}=1$ and we know that the proposition holds
by Lemma \ref{easy_lemma} (\ref{easy_property4}). Assume $k>m $. Let $ k = (q - 1)m + r $, where $q=\ceil{k/m}$ and $ 1 \ls r \ls m $. By Lemmas \ref{transform}, it suffices to prove the following inequality.
\begin{align}
\sum_{i \in S_{j, \ell, r}}C_{(q - 2)m + j, q + d - i - 1} \gs \sum_{i \in S_{j, \ell, r}}C_{(q - 2)m + j, k - d - i - 1}.\label{ssubs1}
\end{align}
\noindent By Lemma \ref{consecutive}, we assume

\[S_{j, \ell, r} = \cb{\min S_{j, \ell, r}, \min S_{j, \ell, r} + 1, \ldots, \max S_{j, \ell, r}}.\]

\noindent Let $ S_1 = \cb{q + d - i - 1: i \in S_{j, \ell, r}} $ and $ S_2 = \cb{k - d - i - 1: i \in S_{j, \ell, r}} $. Then (\ref{ssubs1}) is equivalent to the following inequality.
\begin{align}
\sum_{i \in S_1}C_{(q - 2)m + j, i} \gs \sum_{i \in S_2}C_{(q - 2)m + j, i}.\label{ssubs2}
\end{align}
Let $k'={(q - 2)m + j}$, where $1\ls j\ls m$. Then $k'<k$ and $\ceil{k'/m}=q-1$. By inductive hypothesis, we have \[C_{k', \ceil{k'/m} + d'} \gs C_{k', k' - d'}\]
for any positive integer $ d' $ with $ 2d' \ls k' - \ceil{k'/m} $. Now we have $|S_1|=|S_2|$ and $$\min S_1=q + d - \max S_{j, \ell, r} - 1\ls k - d - \max S_{j, \ell, r} - 1=\min S_2$$ by $ 2d \ls k - \ceil{k/m} $. If we have $\min S_1+\max S_2\gs k'+\ceil{k'/m}$, then (\ref{ssubs2}) holds by Lemma \ref{equiv} and we finish the proof.

Recall $ k = (q - 1)m + r $ and $q=\ceil{k/m}$. Also
 $\min S_1=q + d - \max S_{j, \ell, r} - 1$ and $\max S_2=k - d - \min S_{j, \ell, r} - 1$. If we have $\min S_{j, \ell, r} + \max S_{j, \ell, r} \ls m + r - j - 1$, then $\min S_1+\max S_2\gs k'+\ceil{k'/m}=k'+(q-1)$ holds.
Now we are going to prove
\begin{align}
\min S_{j, \ell, r} + \max S_{j, \ell, r} \ls m + r - j - 1.\label{ssubs3}
\end{align}


\begin{enumerate}
\item If $ r - j - 1 \gs \min S_{j, \ell, r} $, combining with $ m \gs r - 1 \gs \max S_{j, \ell, r} $, then (\ref{ssubs3}) holds.

\item If $ r - j - 1 < \min S_{j, \ell, r} $, then

\[S_{j, \ell, r} = \cb{0 \ls i \ls r - 1: \binom{r - 1}{i} \gs \ell}.\]

 By the symmetry of binomial coefficient, we have $ \min S_{j, \ell, r} + \max S_{j, \ell, r} = r - 1 $, then (\ref{ssubs3}) also holds.\qedhere
\end{enumerate}
\end{proof}

Given positive integers $ k$ and $ m $, recall again $ q := \ceil{k/m} $. By Proposition \ref{weak_spirality} and the same proof of Lemma \ref{equiv}, we have the following corollary.
\begin{cor}
	
Let $ k, m $ be positive integers with $ k, m \gs 2 $. Then $ \alpha(k) \ls \frac{1}{2}\pt{k + q} $.	

\end{cor}

\begin{cor}
	
Let $ k, m $ be positive integers with $ k, m \gs 2 $. If $ n \gs k + q $ and $ q \ls \ell \ls \floor{\frac{n - 1}{2}} $, then $ C_{k, \ell} \gs C_{k, n - \ell} $.
	
\end{cor}

\begin{proof}
By Corollary 2.7, 	$ \alpha(k) \ls \frac{1}{2}\pt{k + q} $.
We prove the result by considering the following two cases.
	
\begin{enumerate}
\item $ q \ls \ell \ls \frac{1}{2}(k + q) $. Let $d=\ell-q$. Then $ 0\ls d \ls \frac{1}{2}\pt{k - q} $. By Proposition \ref{weak_spirality}, we have $C_{k, \ell} \gs C_{k, k + q - \ell}$. Since $k+q-\ell\gs \frac{1}{2}(k + q)\gs \alpha(k)$ and $k+q-\ell\ls n - \ell$, by Proposition \ref{unimodality}, we have $C_{k, k + q - \ell} \gs C_{k, n - \ell}$. Thus $C_{k, \ell} \gs C_{k, n - \ell}.$
		
		
\item $ \frac{1}{2}(k + q) < \ell \ls \floor{\frac{n - 1}{2}} $. Note that $ \alpha(k) \ls \frac{1}{2}(k + q) $ and $ n - \ell > \ell $. By Proposition \ref{unimodality}, we have
$C_{k, \ell} \gs C_{k, n - \ell}.$
\end{enumerate}
	
\end{proof}

\section{Proof of Theorem \ref{main_thm}}\label{proof}

For a set $ S $, denote $ \sP(S) $ as the whole family of non-empty proper subsets of $ S $, that is, $ \sP(S) = 2^S \setminus \{S, \emptyset\} $, where $ 2^S := \cb{T: T \subseteq S} $. For a member $ B \in \sP([n]) $, denote $ B^c := [n] \setminus B $. For a family $ \sB \subseteq \sP([n]) $, denote $ \sB^c := \cb{B^c: B \in \sB} $. For an integer $ \ell $ with $ \ell \in [n - 1] $, denote $ \sB(\ell) := \cb{B \in \sB: \abs{B} = \ell} $.

An intersecting familiy $ \sA \subseteq \binom{[n]_m}{k} $ (resp. $ \sB \subseteq \sP([n]) $) is called \emph{\textbf{maximal}} if for any $ A \in \binom{[n]_m}{k} \setminus \sA $ (resp. $ B \in \sP([n]) \setminus \sB $), we have $ \sA \cup \{A\} $ (resp. $ \sB \cup \{B\} $) is no longer intersecting.

\begin{lem}\label{lem31} Let $ \sB \subseteq \sP\pt{[n]} $
be a maximal intersecting family and  $ B \in \sP\pt{[n]} $. Then we have $ \abs{\cb{B, B^c} \cap \sB} = 1 $.

\end{lem}

\begin{proof}
For any $ B \in \sP([n]) $, we have $ \abs{\cb{B, B^c} \cap \sB} \ls 1 $, by $ \sB $ being intersecting. Suppose $ \abs{\cb{B, B^c} \cap \sB} = 0 $. Since $ \sB \subseteq \sP([n]) $ is maximal intersecting, there are $ B_1, B_2 \in \sB $, such that $ B_1 \cap B = B_2 \cap B^c = \emptyset $, which implies that $ B_1 \subseteq B^c $ and $ B_2 \subseteq B $. Thus $ B_1 \cap B_2 = \emptyset $, a contradiction. Hence $ \abs{\cb{B, B^c} \cap \sB} = 1 $.
\end{proof}

\begin{cor}\label{cor32} Let $ \sB \subseteq \sP\pt{[n]} $ be
a maximal intersecting family and  $ B \in \sB $. Then for any $ B' \in \sP([n]) $ with $ B \subseteq B'$, we have  $ B' \in \sB $.

\end{cor}

\begin{proof}
Suppose $ B \subseteq B_0 $ but $ B_0 \notin \sB $. By Lemma \ref{lem31}, we have $ B_0^c \in \sB $. Note that $ B_0^c \cap B = \emptyset $, which is a contradiction with $ \sB $ being intersecting.
\end{proof}


Let $ k, m, n $ be positive integers with $ k, m \gs 2 $ and $n\ge k+q$, where $ q := \ceil{k/m} $. Then $ n \gs k + q > k \gs 2 $. Thus for any $ A \in \binom{[n]_m}{k} $, we have $ \emptyset \ne A \cap [n] \subsetneqq [n] $. For any $ A \in \binom{[n]_m}{k} $, define $ \varphi(A) := A \cap [n] \in \sP([n]) $. Let $ \sA \subseteq \binom{[n]_m}{k} $. Denote $\varphi(\sA) := \cb{\varphi(A): A \in \sA} $.
By Lemma \ref{lem31}, for any maximal intersecting family $ \sA \subseteq \binom{[n]_m}{k} $, there is a maximal intersecting family $ \sB_{\sA} \subseteq \sP([n]) $, such that $ \sB_{\sA} \supseteq \varphi(\sA) $. Denote $\varphi^{-1}(\sB_{\sA}) := \cb{A \in \binom{[n]_m}{k}: \varphi(A) \in \sB} $.

\begin{lem}\label{lem310} Let $ \sA \subseteq \binom{[n]_m}{k} $
be a maximal intersecting family. Then we have $ \sA = \varphi^{-1}(\sB_{\sA}) $.

\end{lem}

\begin{proof}
\begin{enumerate}
	\item If $ A \in \binom{[n]_m}{k} $ satisfying $ \varphi(A) \in \varphi(\sA) $, then $ A \in \sA $ since $ \sA $ is maximal intersecting. Hence $ \varphi^{-1}\pt{\varphi(\sA)} \subseteq \sA $. Note that $ \sA \subseteq \varphi^{-1}\pt{\varphi(\sA)} $. Thus we have $ \varphi^{-1}\pt{\varphi(\sA)} = \sA $.
	
	\item Suppose $ \varphi^{-1}(\sB_{\sA} \setminus \varphi(\sA)) \ne \emptyset $, say $ X \in \varphi^{-1}(\sB_{\sA} \setminus \varphi(\sA)) $. Then $ \varphi(X) \in \sB_{\sA} $ but $ X \notin \sA $. Since $ \sA $ is maximal intersecting, there is $ Y \in \sA $ such that $ X \cap Y = \emptyset $. Then $ \varphi(X), \varphi(Y) \in \sB_{\sA} $ and $ \varphi(X) \cap \varphi(Y) = \emptyset $,  a contradiction with $ \sB_{\sA} $ being intersecting. Thus we have $ \varphi^{-1}(\sB_{\sA} \setminus \varphi(\sA)) = \emptyset $.
	
	\item Finally, we have
	
	\[\varphi^{-1}(\sB_{\sA}) = \varphi^{-1}\pt{\varphi(\sA) \cup \pt{\sB_{\sA} \setminus \varphi(\sA)}} = \varphi^{-1}(\varphi(\sA)) \cup \varphi^{-1}(\sB_{\sA} \setminus \varphi(\sA)) = \sA.\]
\end{enumerate}
\end{proof}

\begin{lem}\label{lem311} Let $ \sA \subseteq \binom{[n]_m}{k} $
be a maximal intersecting family. Then we have $ \abs{\varphi^{-1}(\sB_{\sA})} = \sum_{\ell = q}^{n - q}C_{k, \ell} \cdot \abs{\sB_{\sA}(\ell)} $.

\end{lem}

\begin{proof}
\noindent By $ n \gs k + q $ and Lemma \ref{easy_lemma} (\ref{easy_property1}), we have $ C_{k, \ell} = 0 $ for $ \ell < q $ or $ \ell > n - q $. Thus
\begin{align*}
\abs{\varphi^{-1}(\sB_{\sA})}{}={}& \sum_{\ell = 1}^{n - 1} \abs{\cb{A \in \binom{[n]_m}{k}: \varphi(A) \in \sB_{\sA}(\ell)}} \\
={}& \sum_{\ell = 1}^{n - 1} \abs{\cb{(m_1, \cdots, m_\ell) \in [m]^\ell: m_1 + \cdots + m_\ell = k}} \cdot \abs{\sB_{\sA}(\ell)} \\
={}&\sum_{\ell = q}^{n - q}C_{k, \ell} \cdot \abs{\sB_{\sA}(\ell)}.
\end{align*}
\end{proof}

Define

\[\sU := \cb{B \in \sP([n]): 1 \in B}.\]

\noindent Then $ \sF^{(m)}_{0, 1} = \varphi^{-1}(\sU) $ by Lemma \ref{lem310}. Also, when $ q \ls \ell \ls \floor{\frac{n - 1}{2}} $, $ \sU(\ell) $ is the maximum intersecting family in $ \binom{[n]}{\ell} $ by the Erd\H{o}s-Ko-Rado Theorem.

\begin{lem}\label{lem312} Let $ \sA \subseteq \binom{[n]_m}{k} $
be a maximal intersecting family. Then we have $ |\varphi^{-1}(\sU)|\gs |\varphi^{-1}(\sB_{\sA})| $.

\end{lem}

\begin{proof}
We denote $ \sD_{\sA} := \sU \setminus \sB_{\sA} $. Then we have $ \sB_{\sA} = (\sU \setminus \sD_{\sA}) \cup \sD_{\sA}^c $ by Lemma \ref{lem31}.  Note that
\[\abs{\sU(\ell)} \gs \abs{\sB_{\sA}(\ell)} =\abs{\sU(\ell)} - \abs{\sD_{\sA}(\ell)} + \abs{\sD_{\sA}^c(\ell)} = \abs{\sU(\ell)} - \abs{\sD_{\sA}(\ell)} + \abs{\sD_{\sA}(n - \ell)} \]

\noindent which implies $ \abs{\sD_{\sA}(\ell)} \gs \abs{\sD_{\sA}(n - \ell)} $. By Corollary 2.8, $C_{k, \ell} \gs C_{k, n - \ell}$ for $q\ls \ell \ls \floor{\frac{n - 1}{2}}$. So
by Lemma \ref{lem311}, we have
\begin{align*}
&\abs{\varphi^{-1}(\sU)} - \abs{\varphi^{-1}(\sB_{\sA})}\\
={}&\sum_{\ell = q}^{n - q}C_{k, \ell} \cdot \pt{\abs{\sU(\ell)} - \abs{\sB_{\sA}(\ell)}} \\
={}&\sum_{\ell = q}^{n - q} C_{k, \ell}\cdot  \pt{\abs{\sD_{\sA}(\ell)} - \abs{\sD_{\sA}(n - \ell)}} \\
={}&\sum_{\ell = q}^{\floor{\frac{n - 1}{2}}} C_{k, \ell} \cdot \pt{\abs{\sD_{\sA}(\ell)} - \abs{\sD_{\sA}(n - \ell)}} + \sum_{\ell = \ceil{\frac{n + 1}{2}}}^{n - q} C_{k, \ell} \cdot \pt{\abs{\sD_{\sA}(\ell)} - \abs{\sD_{\sA}(n - \ell)}} \\
={}&\sum_{\ell = q}^{\floor{\frac{n - 1}{2}}}  \pt{C_{k, \ell} - C_{k, n - \ell}} \cdot \pt{\abs{\sD_{\sA}(\ell)} - \abs{\sD_{\sA}(n - \ell)}} \\
\gs{}&0.
\end{align*}
\end{proof}

\begin{lem}\label{lemma_the_l} Let $ \sA \subseteq \binom{[n]_m}{k} $
be a maximal intersecting family.
If there is an integer $ \ell_0 $ with $ q \ls \ell_0 \ls \floor{\frac{n - 1}{2}} $ such that $ \abs{\sD_{\sA}(\ell_0)} = \abs{\sD_{\sA}(n - \ell_0)} $, then for any positive integer $ \ell $ with $ \ell_0 \ls \ell \ls n - 1 $, we have $ \abs{\sD_{\sA}(\ell)} = \abs{\sD_{\sA}(n - \ell)} $.	
\end{lem}

\begin{proof}
Note that $ n > 2\ell_0 $. Recall $\abs{\sB_{\sA}(\ell)}  = \abs{\sU(\ell)} - \abs{\sD_{\sA}(\ell)} + \abs{\sD_{\sA}(n - \ell)}$.
If $ \abs{\sD_{\sA}(\ell_0)} = \abs{\sD_{\sA}(n - \ell_0)} $, then $ \abs{\sB_{\sA}(\ell_0)} = \abs{\sU(\ell_0)} $. Thus $ \sB_{\sA}(\ell_0) $ is a trivial maximum intersecting family by the Erd\H{o}s-Ko-Rado Theorem. Without loss of generality, we may assume that

\[\sB_{\sA}(\ell_0) = \cb{B \in \binom{[n]}{\ell_0}: x \in B}\]

\noindent for some fixed $ x \in [n] $. Suppose $ \ell $ is a positive integer with $ \ell_0 \ls \ell \ls n - 1 $. For any $ B' \in \binom{[n]}{\ell} $ with $ x \in B' $, there is some $ B \in \binom{[n]}{\ell_0} $ with $ x \in B \subseteq B' $. By Corollary \ref{cor32}, we have

\[ \sB_{\sA}(\ell) = \cb{B' \in \binom{[n]}{\ell}: x \in B'}, \]

\noindent that is, $ \sB_{\sA}(\ell) $ is a trivial maximum intersecting family for $ \ell_0 \ls \ell \ls n - 1 $ by the Erd\H{o}s-Ko-Rado Theorem. Thus $ \abs{\sB_{\sA}(\ell)} = \abs{\sU(\ell)} $, and so $ \abs{\sD_{\sA}(\ell)} = \abs{\sD_{\sA}(n - \ell)} $.		
\end{proof}

\begin{cor}\label{cor} Let $ \sA \subseteq \binom{[n]_m}{k} $
be a maximal intersecting family.	
If $ \abs{\sD_{\sA}(q)} = \abs{\sD_{\sA}(n - q)} $, then $ \varphi^{-1}(\sB_{\sA}) \cong \varphi^{-1}(\sU) $.	

\end{cor}

\begin{proof}
	
If $ \abs{\sD_{\sA}(q)} = \abs{\sD_{\sA}(n - q)} $, then by Lemma \ref{lemma_the_l}, we have $ \abs{\sD_{\sA}(\ell)} = \abs{\sD_{\sA}(n - \ell)} $ for $ q \ls \ell \ls n - 1 $. Moreover, we have

\[\bigcup_{\ell = q}^{k} \sB_{\sA}(\ell) \cong \bigcup_{\ell = q}^{k} \sU(\ell).\]

\noindent Thus $ \varphi^{-1}(\sB_{\sA}) \cong \varphi^{-1}(\sU) $.\qedhere

\end{proof}

\noindent{\bf Proof of Theorem \ref{main_thm}} If $m=1$, Theorem \ref{main_thm} is the famous Erd\H{o}s-Ko-Rado Theorem. So we will assume $m\gs 2$ in the following proof. Let $ \sA \subseteq \binom{[n]_m}{k} $
be a maximal intersecting family. By Lemmas \ref{lem310} and \ref{lem312}, we have $ \abs{\sA}= \abs{\varphi^{-1}(\sB_{\sA})}\ls \abs{\varphi^{-1}(\sU)}=\abs{\sF_{0, 1}^{(m)}} $.

Now we consider the cases {\rm(a)} $ n > k + q $, or {\rm(b)} $ n = k + q $, $ k > m $ and $ m \nmid k $ and assume $ \abs{\sA} = \abs{\sF_{0, 1}^{(m)}}$. We will show that  $ \sA \cong \sF_{0, 1}^{(m)} $. If $ \abs{\sD_{\sA}(q)} = \abs{\sD_{\sA}(n - q)} $, then the result holds by Corollary \ref{cor}. Suppose $ \abs{\sD_{\sA}(q)} \ne \abs{\sD_{\sA}(n - q)} $.
 Take
\[ L := \max\cb{\ell \in \BZ: q \ls \ell \ls \floor{\frac{n - 1}{2}} \text{ and } \abs{\sD_{\sA}(\ell)} - \abs{\sD_{\sA}(n - \ell)} > 0}. \]

\noindent By Lemma \ref{lemma_the_l}, we have $ \abs{\sD_{\sA}(\ell)} > \abs{\sD_{\sA}(n - \ell)} $ for $ q \ls \ell \ls L $. Hence

\begin{equation}\label{eqn_with_L}
\abs{\varphi^{-1}(\sU)} - \abs{\varphi^{-1}(\sB_{\sA})} = \sum_{\ell = q}^{L}  \pt{C_{k, \ell} - C_{k, n - \ell}} \cdot \pt{\abs{\sD_{\sA}(\ell)} - \abs{\sD_{\sA}(n - \ell)}}.
\end{equation}

\noindent Since $ \abs{\sA} = \abs{\sF_{0, 1}^{(m)}}$ (that is $\abs{\varphi^{-1}(\sB_{\sA})}=\abs{\varphi^{-1}(\sU)} $), we have $ C_{k, \ell} = C_{k, n - \ell} $ for $ q \ls \ell \ls L $. In particular, $ C_{k, q} = C_{k, n - q} $. By Lemma \ref{easy_lemma} (\ref{easy_property1}), we have $ n - q \ls k $ which implies $ n = k + q $. In this case, $ C_{k, q} = C_{k, k} = 1 $ by Lemma \ref{easy_lemma} (\ref{easy_property3}). Then we have $ q = 1 $, or { $ q > 1 $ and $ m \mid k $ by Lemma \ref{easy_lemma} (\ref{easy_property2}), a contradiction with our conditions. Thus, when {\rm(a)} $ n > k + q $, or {\rm(b)} $ n = k + q $, $ k > m $ and $ m \nmid k $, $ \abs{\sA} = \abs{\sF_{0, 1}^{(m)}} $ if and only if $ \sA \cong \sF_{0, 1}^{(m)} $. Hence we complete the proof. \QEDopen

\vskip.4cm

\begin{rem} {\em When $ n = k + q=k+\ceil{k/m} $, and $ k \ls m $ or  $ m \mid k $, 
 families attaining the maximum size are not limited to those isomorphic
 to $\sF_{0, 1}^{(m)}$. This result holds obviously when $m=1$. Now we  construct a maximal intersecting family $ \sN \subseteq \sP([n]) $ such that $ \abs{\varphi^{-1}(\sN)} = \abs{\varphi^{-1}(\sU)} = \abs{\sF_{0, 1}^{(m)}} $ but $ \varphi^{-1}(\sN) $ is non-trivial for $m\ge 2$. By Lemma \ref{easy_lemma} (\ref{easy_property2}) and (\ref{easy_property3}),  we have $ C_{k, q} = C_{k, k} = 1 $ in this case. Let $ \sN := (\sU \setminus \sD_{\varphi^{-1}(\sN)}) \cup \sD_{\varphi^{-1}(\sN)}^c $, where $\sD_{\varphi^{-1}(\sN)}:= 2^{[q]} \cap \sU $. Then 
  $ \sU \setminus \sN = \sD_{\varphi^{-1}(\sN)} $ and $ \abs{\sD_{\varphi^{-1}(\sN)}(s)} \ne 0 $ if and only if $ 1 \ls s \ls q $. Since $m\gs 2$, we have $k>q$. Then $\abs{\sD_{\varphi^{-1}(\sN)}(q)}-\abs{\sD_{\varphi^{-1}(\sN)}(k)}=\abs{\sD_{\varphi^{-1}(\sN)}(q)}\ne 0 $. If $ q < \ell \ls \floor{\frac{n - 1}{2}} $, we have $\abs{\sD_{\varphi^{-1}(\sN)}(\ell)}=0$ and $\abs{\sD_{\varphi^{-1}(\sN)}(n-\ell)}=0$ by $n-\ell>\ell>q$.   
  Thus for  $ \ell $ with $ q \ls \ell \ls \floor{\frac{n - 1}{2}} $, we have $ \abs{\sD_{\varphi^{-1}(\sN)}(\ell)} - \abs{\sD_{\varphi^{-1}(\sN)}(n - \ell)} = \abs{\sD_{\varphi^{-1}(\sN)}(\ell)} - \abs{\sD_{\varphi^{-1}(\sN)}(k + q - \ell)} \ne 0 $ if and only if $ \ell = q $. Hence by (\ref{eqn_with_L}) we have
\[\abs{\varphi^{-1}(\sU)} - \abs{\varphi^{-1}(\sN)} = \pt{C_{k, q} - C_{k, k}} \cdot \pt{\abs{\sD_{\varphi^{-1}(\sN)}(q)} - \abs{\sD_{\varphi^{-1}(\sN)}(k)}} = 0.\]

 Next we claim that $ \sN \subseteq \sP([n]) $ is intersecting. Let $ G_1, G_2 \in \sN$. We can assume $\abs{\{G_1,G_2\}\cap \sU}\ls 1$.

\begin{enumerate}
	\item If $ G_1, G_2 \in \sD_{\varphi^{-1}(\sN)}^c $, then we have $ G_1^c, G_2^c \in \sD_{\varphi^{-1}(\sN)} $. So $ G_1^c \cup G_2^c \subseteq [q] \subsetneqq [n] $. Thus $ G_1 \cap G_2 = (G_1^c \cup G_2^c)^c \ne \emptyset $.
	
	\item Assume $ G_1 \in \sU \setminus \sD_{\varphi^{-1}(\sN)} $ and $ G_2 \in \sD_{\varphi^{-1}(\sN)}^c $. If $ G_1 \cap G_2 = \emptyset $, then $ G_1 \subseteq G_2 ^c$, which implies $ G_1 \in \sD_{\varphi^{-1}(\sN)} $, a contradiction.
\end{enumerate}

\noindent Hence $ \sN $ is intersecting. In particular, $ \varphi^{-1}(\sN) \subseteq \binom{[n]_m}{k} $ is intersecting. Next, we claim that $ \varphi^{-1}(\sN) $ is non-trivial. Note that $ [q]^c = \cb{q + 1, \ldots, n} \in \varphi^{-1}(\sN) $. For any integer $ i $ with $ q + 1 \ls i \ls n $, we have

\[\cb{1} \cup [q + 1, n] \setminus \{i\} \in \varphi^{-1}(\sN).\]

\noindent Thus $ \varphi^{-1}(\sN) $ is non-trivial.}
\end{rem}

\section*{Acknowledgement}
This research was supported by   the National Natural Science Foundation of China (Grant 12171272 \&
12161141003).

\end{document}